\documentclass[12pt,reqno]{amsart}
\usepackage{amsmath}
\usepackage{amssymb, gensymb}
\usepackage[left=3.7cm,top=3cm,right=3.7cm,bottom=3cm]{geometry}
\usepackage[usenames,dvipsnames]{colo	r}
\usepackage{graphicx}
\usepackage{caption}
\usepackage{subcaption}

\newcommand{\whp}{{whp}}
\newcommand{\codeg}{\overline{\deg}}

\begin{document}
\newtheorem{theorem}{Theorem}[section]
\newtheorem{lemma}[theorem]{Lemma}
\newtheorem{definition}[theorem]{Definition}
\newtheorem{conjecture}[theorem]{Conjecture}
\newtheorem{proposition}[theorem]{Proposition}
\newtheorem{algorithm}[theorem]{Algorithm}
\newtheorem{corollary}[theorem]{Corollary}
\newtheorem{observation}[theorem]{Observation}
\newtheorem{problem}[theorem]{Open Problem}
\newtheorem{remark}[theorem]{Remark}
\newcommand{\noin}{\noindent}
\newcommand{\ind}{\indent}
\newcommand{\om}{\omega}
\newcommand{\I}{\mathcal I}
\newcommand{\N}{{\mathbb N}}
\newcommand{\LL}{\mathbb{L}}
\newcommand{\R}{{\mathbb R}}
\newcommand{\Z}{{\mathbb Z}}

\newcommand{\E}{\mathbb E}
\newcommand{\Prob}{\mathbb{P}}
\newcommand{\pr}{\mathbb{P}}
\newcommand{\eps}{\varepsilon}
\newcommand{\G}{{\mathcal{G}}}
\newcommand{\Bin}{\mathrm{Bin}}
\newcommand{\Po}{\mathrm{Po}}

\newcommand{\cF}{{\mathcal F}}

\newcommand{\RG} {\ensuremath{\mathcal G(n,r)}}
\newcommand{\RGuv} {\ensuremath{\widetilde{\mathcal G}(n,r)}}
\newcommand{\SR} {\ensuremath{\mathcal S_n}}
\newcommand{\SRn} {\ensuremath{\mathcal S_n}}

\newcommand{\x}{{\mathbf x}}
\newcommand{\y}{{\mathbf y}}
\newcommand{\z}{{\mathbf z}}

\newcommand{\be}{{\mathbf e}}
\newcommand{\bu}{{\mathbf u}}
\newcommand{\bv}{{\mathbf v}}
\newcommand{\bw}{{\mathbf w}}


\title{Clique colourings of geometric graphs}

\author{Colin McDiarmid}
\address{Department of Statistics, Oxford University, 24 - 29 St Giles, Oxford OX1 3LB, UK}
\email{\tt cmcd@stats.ox.ac.uk}

\author{Dieter Mitsche}
\address{Universit\'{e} de Nice Sophia-Antipolis, Laboratoire J-A Dieudonn\'{e}, Parc Valrose, 06108 Nice cedex 02}
\email{\texttt{dmitsche@unice.fr}}

\author{Pawe\l{} Pra\l{}at}
\address{Department of Mathematics, Ryerson University, Toronto, ON, Canada}
\email{\tt pralat@ryerson.ca}

\date{30 June 2017, v3}

\keywords{geometric graphs, random geometric graphs, clique chromatic number}
\thanks{This work was initiated at the Centre de Recerca Matematica (CRM) in Barcelona within the ``Research programme on Algorithmic Perspective in Economics and Physics''. The research of the third author is supported in part by NSERC and Ryerson University.}
\subjclass{05C80, 
05C15, 
05C35. 
}

\maketitle

\begin{abstract}
A \emph{clique colouring} of a graph is a colouring of the vertices such that no maximal clique is monochromatic (ignoring isolated vertices). The least number of colours in such a colouring is the \emph{clique chromatic number}.  Given $n$ points $\x_1, \ldots,\x_n$ in the plane, and a threshold $r>0$, the corresponding geometric graph has vertex set $\{v_1,\ldots,v_n\}$, and distinct $v_i$ and $v_j$ are adjacent when the Euclidean distance between $\x_i$ and $\x_j$ is at most $r$. We investigate the clique chromatic number of such graphs.

We first show that the clique chromatic number is at most 9 for any geometric graph in the plane, and briefly consider geometric graphs in higher dimensions. Then we study the asymptotic behaviour of the clique chromatic number for the random geometric graph $\RG$ in the plane, where $n$ random points are independently and uniformly distributed in a suitable square. We see that as $r$ increases from 0, with high probability the clique chromatic number is 1 for very small $r$, then 2 for small $r$, then at least 3 for larger $r$, and finally drops back to 2.
\end{abstract}


\section{Introduction and main results}\label{sec:intro}

In this section we introduce clique colourings and geometric graphs; and we present our main results, on clique colourings of deterministic and random geometric graphs.

Recall that a \emph{proper colouring} of a graph is a labeling of its vertices with colours such that no two vertices sharing the same edge have the same colour; and the smallest number of colours in a proper colouring of a graph $G=(V,E)$ is its \emph{chromatic number}, denoted by $\chi(G)$.
 
We are concerned here with another notion of vertex colouring. A \emph{clique} $S \subseteq V$ is a subset of the vertex set such that each pair of vertices in $S$ is connected by an edge; and a clique is \emph{maximal} if it is not a proper subset of another clique. A \emph{clique colouring} of a graph $G$ is a colouring of the vertices such that no maximal clique is monochromatic, ignoring isolated vertices.  The least number of colours in such a colouring is the \emph{clique chromatic number} of $G$, denoted by $\chi_c(G)$. (If $G$ has no edges we take $\chi_c(G)$ to be 1.)  Clearly, $\chi_c(G) \le \chi(G)$ but it is possible for $\chi_c(G)$ to be much smaller than $\chi(G)$. For example, for any $n \geq 2$ we have $\chi(K_n) = n$ but $\chi_c(K_n) = 2$.  Note that if $G$ is triangle-free then $\chi_c(G)=\chi(G)$. 

A standard example of a hypergaph arising from a graph $G$ is the hypergraph $H$ with vertex set $V(G)$ and edges the vertex sets of the maximal cliques.  A clique-colouring of $G$ is exactly the standard hypergraph colouring of $H$, that is, colouring the vertices so that no edge is monochromatic.
\smallskip

For several graph classes the maximum clique chromatic number is known to be 2 or 3. For maximum value 2 we have for example: comparability graphs~\cite{Duffus+},  claw-free perfect graphs~\cite{BGGPS2004}, odd-hole and co-diamond free graphs~\cite{Defossez}, claw-free planar graphs~\cite{SLK2014}, powers of cycles (other than odd cycles longer than three, which need three colours)~\cite{Campos}, and claw-free graphs with maximum degree at most $7$ (again, except for odd cycles of length more than three)~\cite{Liang}. For maximum value 3 we have for example: planar graphs~\cite{Mohar}, co-comparability graphs~\cite{Duffus+}, circular-arc graphs (see~\cite{Cerioli}) and generalised split graphs (see~\cite{Gravier}).  Further related results can be found in~\cite{Andreae},~\cite{Gravier} and~\cite{Klein}. It was believed for some time that perfect graphs had bounded clique chromatic number, perhaps with maximum value 3 (see~\cite{Duffus+} or for example~\cite{JensenToft}); but it was shown very recently that in fact such clique chromatic numbers are unbounded~\cite{cptt2006}. The behaviour of the clique chromatic number for the binomial (known also as Erd\H{o}s-R\'{e}nyi) random graph $\G(n,p)$ is investigated in~\cite{mmp-gnp} and~\cite{ak-gnp}.

On the algorithmic side, it is known that testing whether $\chi_c(G)=2$ for a planar graph can be performed in polynomial time~\cite{Kratochvil}, but deciding whether $\chi_c(G)=2$ is $NP$-hard for perfect graphs~\cite{Kratochvil} and indeed for $K_4$-free perfect graphs~\cite{Defossez}, and for graphs with maximum degree $3$~\cite{BGGPS2004}; see also~\cite{marx11}.
\smallskip

We are interested here primarily in clique colourings of geometric graphs in the plane, but we shall also briefly consider geometric graphs in $\R^d$ for any positive integer $d$.  Given $n$ points $\x_1,\ldots,\x_n$ in $\R^d$  and given a threshold distance $r>0$, the corresponding {(Euclidean) geometric graph} has vertex set $\{v_1,\ldots,v_n\}$, and for $i \neq j$, vertices $v_i$ and $v_j$ are adjacent when the Euclidean distance $d(\x_i, \x_j) \leq r$. We call a graph $G$ \emph{geometric} or \emph{geometric in} $\R^d$ if there are points $\x_j$ and $r>0$ realising $G$ as above.  By rescaling by a factor $1/r$ we may assume, without loss of generality, that $r=1$. A geometric graph in $\R^2$ is also called a \emph{unit disk graph}.

\smallskip

Our first theorem shows that the clique chromatic number is uniformly bounded for geometric graphs in the plane.  (In contrast, Bacs\'{o} et al.~\cite{BGGPS2004} observed that $\chi_c(G)$ is unbounded even for line graphs of complete graphs, and recall that $\chi_c(G)$ is unbounded for perfect graphs.)

\begin{theorem}\label{thm.detR2}
If $G$ is a geometric graph in the plane $\R^2$ then $\chi_c(G) \le 9$.
\end{theorem}
Let $\chi_c^{\max}(\R^d)$ denote the maximum value of $\chi_c(G)$ over geometric graphs $G$ in $\R^d$.  Clearly $\chi_c^{\max}(\R^2)$ is at least 3 (consider $C_5$) so we have $ 3 \leq \chi_c^{\max}(\R^2) \leq 9$: it would be interesting to improve these bounds.  In Section~\ref{sec.detproofs} we shall see that more generally $\chi_c^{\max}(\R^d)$ is finite for each~$d$, but (perhaps unsurprisingly) $\chi_c^{\max}(\R^d) \to \infty$ as $d \to \infty$; and we shall see further related deterministic results.

\smallskip

For random geometric graphs the upper bound in Theorem~\ref{thm.detR2} can often be improved. Given a positive integer $n$ and a threshold distance $r>0$, we consider the \emph{random geometric graph} $G \in \RG$ on vertex set $V=\{v_1,\ldots,v_n\}$ obtained as before by starting with $n$ random points sampled independently and uniformly in the square $\SR = \left[-\sqrt{n}/2,\sqrt{n}/2\right]^2$, see~\cite{Penrose}.  (We could equally work with the unit square $[0,1]^2$.) Note that, with probability $1$, no point in $\SR$ is chosen more than once, so we may identify each vertex $v \in V$ with its corresponding geometric position $\bv = (v_x,v_y) \in \SRn$. The (usual) chromatic number of $\RG$ was studied in~\cite{cmcd2003,ColinTobias}, see also~\cite{Penrose}.

We say that events $A_n$ hold \emph{with high probability} (\emph{\whp}) if the probability that $A_n$ holds tends to $1$ as $n$ goes to infinity.  Also, we use $\log$ to denote natural logarithm.  It is known that the value $r_c = r_c(n)=\sqrt{(\log n)/\pi}$ is a sharp threshold function for connectivity for $G \in \RG$ (see, for example, \cite{Penrose97,Goel05}). This means that for every $\varepsilon>0$, if $r\le(1-\varepsilon)r_c$, then $G$ is disconnected \whp, whilst if $r\ge (1+\varepsilon)r_c$, then $G$ is connected \whp. 

\smallskip

The next two results summarise what we know about the clique chromatic number $\chi_c$ of a random geometric graph $G$ in the plane; but first here is an overview. As $r$ increases from 0 we have \whp\ the following rough picture: $\chi_c(G)$ is 1 up to about $n^{-1/2}$, then 2 up to about $n^{-1/8}$, then at least 3 (and at most $\chi_c^{\max}(\R^2) \leq 9$) up to about $\sqrt{\log n}$ (roughly the connectivity threshold), when it drops back to 2 and remains there.  

\begin{theorem}\label{thm.RGGsummary}
For the random geometric graph $G \in \RG$ in the plane:
\begin{enumerate}
\item if $nr^2 \to 0$ then $\chi_c=1$ \whp,
\item if $nr^2 \to c$ then $\pr(\chi_c=1) \to e^{-(\pi/2) c}$ and $\pr(\chi_c=2) \to 1- e^{-(\pi/2)c}$,
\item if  $nr^2 \to \infty$ and $nr^8 \to 0$ then $\chi_c=2$ \whp,
\item if  $nr^8 \to c$ then $\pr(\chi_c=2) \to e^{- \mu c}$ and $\pr(\chi_c=3) \to 1- e^{- \mu c}$, for a suitable constant $\mu =\mu(C_5)  >0$ (see below),
\item if  $nr^8 \to \infty$ and  $r \leq 0.46 \sqrt{\log n}$ then $\chi_c \geq 3$ \whp,
\item  if $r \geq 9.27 \, \sqrt{\log n}$ then $\chi_c=2$ \whp.
\end{enumerate}
\end{theorem} 

The constant $\mu$ in part (4) above may be expressed explicitly as an integral, see equation (3.2) in~\cite{Penrose}.  It is the asymptotic expected number of components $C_5$ in the case when $nr^8 \to 1$. We can say more within the interval in (5) above where $\chi_c(G) \geq 3$: at the low end of the interval we have $\chi_c(G)=3$ \whp; and higher up, within a suitable subinterval, $\chi_c(G)$ is \whp\ as large as is possible for a geometric graph.
\begin{proposition} \label{prop.RGG4}
For the random geometric graph $G \in \RG$ in the plane:
\begin{enumerate}
\item if $nr^8 \to \infty$ and $nr^{18} \to 0$ then $\chi_c(G)=3$ \whp,
\item there exists $\varepsilon>0$ such that, if $n^{-\varepsilon} \leq r \leq \varepsilon \sqrt{\log n}$ then $\chi_c(G)= \chi_c^{\max}(\R^2)$ \whp.
\end{enumerate}
\end{proposition}

 The only random geometric graphs we consider here are those described above, where the points are independently and uniformly distributed over a square in the plane.  See~\cite{Penrose} for more general models of random geometric graphs, and see~\cite{dglu2011} in particular for models in high dimensions.


\section{Deterministic results}\label{sec.detproofs}

In this section, we start by proving Theorem~\ref{thm.detR2}, and then consider geometric graphs in dimensions greater than 2.  
After that we give Lemma~\ref{lem.detsmall}, concerning the maximum value of $\chi_c(G)$ for general $n$-vertex graphs, for small values of $n$: this result  will be used in the next section in the proof of Proposition~\ref{prop.RGG4}.


\begin{proof}[Proof of Theorem~\ref{thm.detR2}]
Fix $y$ with $\frac12 < y < \sqrt{3}/2$. Divide the plane into horizontal strips $\R \times [ny,(n+1)y)$ for $n \in \Z$. Suppose we are given a finite set of points in the plane, and let $G$ be the corresponding unit disk graph. Consider one strip, let $W$ be the subset of the given points which are in the strip (which we may assume is non-empty), and $H$ be the geometric graph corresponding to $W$. We claim that $\chi_c(H) \leq 3$.

For $\bu, \bv \in W$ we write $\bu \prec \bv$ if $u_x < v_x$ and $uv \in E(\overline{H})$.  If $\bu \prec \bv$ then $1< d(\bu,\bv) < (v_x-u_x)^2+ \tfrac34$ so $v_x > u_x+\tfrac12$.  Thus if also $\bv \prec \bw$ then $w_x > v_x+ \tfrac12 > u_x+1$, so $\bu \prec \bw$.  Thus $\prec$ is a (strict) partial order on $W$. 
Further, $\overline{H}$ is the corresponding comparability graph, since if $uv \in E(\overline{H})$ then $u_x \neq v_x$ (for if $u_x = v_x$ then $d(\bu,\bv)=|u_y-v_y| <y<1$ so $uv$ is in $E(H)$ not $E(\overline{H})$).  Thus $H$ is a co-comparability graph. Hence, by the result of Duffus et al.~\cite{Duffus+} mentioned earlier, we have $\chi_c(H) \leq 3$, as claimed. (Indeed, we do not know an example where $\chi_c(H)>2$.)

Now label the strips cyclically $a,b,c,a,b,c,a,\ldots$ moving upwards say, and use 3 colours to properly clique colour the $a$-strips, a new set of 3 colours for the $b$-strips and similarly a new set of 3 colours for the $c$-strips, using 9 colours in total.  A monochromatic maximal clique with at least 2 vertices could not have points in two different strips since $2y> 1$, and could not be contained in one strip since we have a proper clique-colouring there.  Thus $\chi_c(G) \leq 9$.
\end{proof}
\smallskip

Theorem~\ref{thm.detR2} shows that the clique chromatic number is at most 9 for any geometric graph in the plane.  We next see that, for a given dimension $d$, there is a uniform bound on the clique chromatic number for all geometric graphs in $\R^d$.

\begin{proposition} \label{thm.detRd}
Let $G$ be a geometric graph in $\R^d$.  Then
\[ \chi_c(G) \leq 2 \, {(\lceil \sqrt{d}\rceil  + 1)}^d < 2\, e^{2\sqrt{d}} \, d^{d/2}.\]
\end{proposition}
Our simple proof uses a tessellation into small hypercubes which induce cliques.  In the case $d=2$ it is better to use hexagonal cells, and then the bound improves from 18 to 14.  In~\cite{MPT}, hexagonal cells are used in pairs to show that $\chi_c^{\max}(\R^2) \leq 10$, nearly matching the upper bound 9 in Theorem~\ref{thm.detR2}.

\begin{proof}
We may assume that the threshold distance $r$ is 1. Let $k= \lceil \sqrt{d}\rceil$, let $s=1/k$, and let $Q$ be the hypercube $[0,s)^d$.  Observe that $Q$ has diameter $s \sqrt{d} \leq 1$, so the subgraph of $G$ induced by the points in $Q$ is complete.  We partition $\R^d$ into the family of translates $Q+ s\z$ of $Q$, for $\z \in \Z^d$.  (Here $Q +\y$ is the set of all points $\x+\y$ for $\x \in Q$.) Consider the subfamily $\cF_0 = (Q+ (k+1)s\z : \z \in \Z^d)$. Let $\z$ and $\z'$ be distinct points in $\Z^d$, and let $\x$ and $\x'$ be points in the cells $Q+(k+1)s\z$ and $Q+(k+1)s\z'$ in $\cF_0$ respectively. Without loss of generality, we may assume that $z_1 > z'_1$.  Then 
\[ d(\x, \x') \geq  x_1-x'_1 
> (k+1)s(z_1-z'_1) - s \geq ks=1. \]
Thus the subgraph $G'$ of $G$ induced on the vertices corresponding to the points in the cells of $\cF_0$ consists of disjoint cliques, with no edges between them.  Hence $\chi_c(G') \leq 2$, since we just need to ensure that each cell with at least two points gets two colours. Finally, let $\cF(\y)$ denote the translate by $\y$ of the family $\cF_0$, so 
\[ \cF(\y) =  (Q +\y + (k+1)s\z: \z \in \Z^d)\]
(and $\cF_0=\cF({\mathbf 0})$).  Let $S(\y)$ be the union of the cells in $\cF(\y)$, and let $G(\y)$ be the subgraph of $G$ induced by the vertices corresponding to the points in $S(\y)$.  Then the $(k+1)^d$ sets $S(\y)$ for $\y \in \{0,\ldots,k\}^d$ partition $\R^d$; and so
\[ \chi_c(G) \leq \sum_{\y \in \{0,\ldots,k\}^d} \chi_c(G(\y)) \leq 2 (k+1)^d, \]
as required for the first inequality.  For the second inequality,  we have
\[ (k+1)^d < (\sqrt{d}+2)^d = d^{d/2} (1+ 2/\sqrt{d})^d < d^{d/2} e^{2\sqrt{d}},\]
and the proof is finished.
\end{proof}

For example, we may deduce from this result that $\chi_c^{\max}(\R^3) \leq 2 \cdot 3^3 = 54$.  It is not hard to make small improvements for each $d$, but let us focus on the case $d=3$.

\begin{proposition} \label{thm.detR3}
If $G$ is a geometric graph in $\R^3$ then $\chi_c(G) \leq 21$.
\end{proposition}

\begin{proof}
Let $T$ denote the unit triangular lattice in $\R^2$, with vertices the integer linear combinations of ${\bf p}=(1,0)$ and ${\bf q}=(\tfrac12, \tfrac{\sqrt{3}}2)$ (and where the edges have unit length). Consider the hexagonal packing in the plane, as in Figure 1, formed from the hexagonal Voronoi cells of $T$. 
\begin{figure}[h]
\begin{center}
\includegraphics[width=2.5in,height=2.5in]{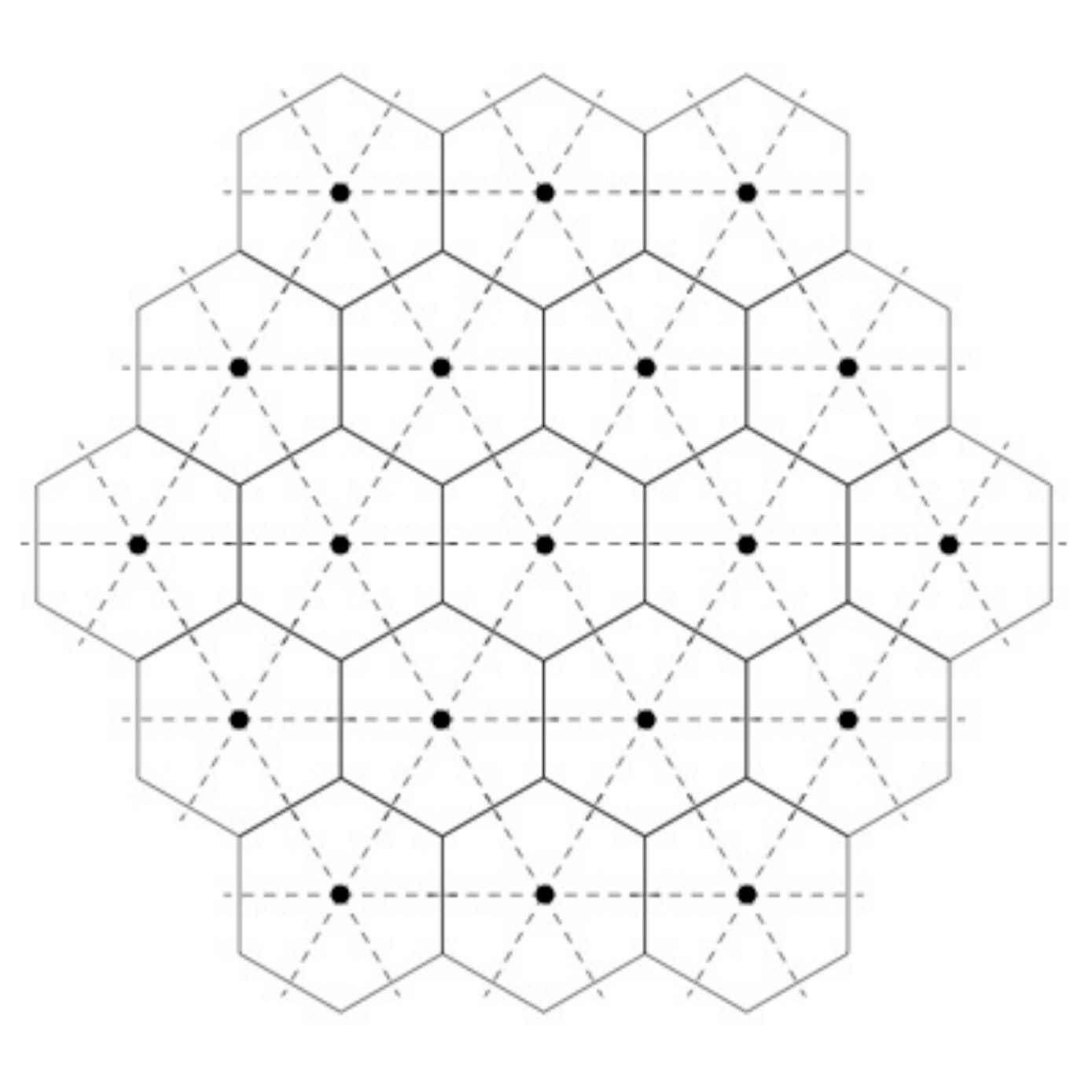}
\end{center}
\caption{Dashed lines join points of the unit triangular lattice $T$ at distance 1, and solid lines bound the hexagonal Voronoi cells}
\label{fig2}
\end{figure}

The sublattice $T'$ of $T$ with vertices generated by $2{\bf p}+{\bf q}$ and ${\bf -p}+3{\bf q}$ is a triangular lattice with edge-length $\sqrt{7}$, and 7 translates of $V(T')$ partition $V(T)$
(for example translate by $(0,0), {\bf q},2{\bf q}, 3{\bf q},{\bf p}\!+\!{\bf q}, {\bf p}\!+\!2{\bf q}, {\bf p}\!+\!3{\bf q}$ -- see Figure~\ref{fig:7colors}, and for example~\cite{MR99}). 
We thus obtain a 7-colouring of the vertices of $T$, and this gives a 7-colouring of the cells.

Since the cells have diameter $2/\sqrt{3}$, the distance between any two cells centred on distinct points in $T'$ is at least $\sqrt{7} - 2/\sqrt{3} \approx 1.491051$.  (In fact, the minimum distance occurs for example between the cells centred on $(0,0)$ and on $2{\bf p}+{\bf q}$, and equals $d( (\tfrac12,\tfrac1{2 \sqrt{3}}),  (2, \tfrac1{\sqrt{3}}) )= \sqrt{\tfrac73} \approx 1.527525$.)
%
%
Thus our 7-colouring of the cells is such that, for any two distinct cells of the same colour, the distance between them is at least $1.49$
(see also Theorems 3 and 4 of~\cite{MR99} for related results).
\begin{figure}[h]
  \includegraphics[width=2.5in,height=2.5in]{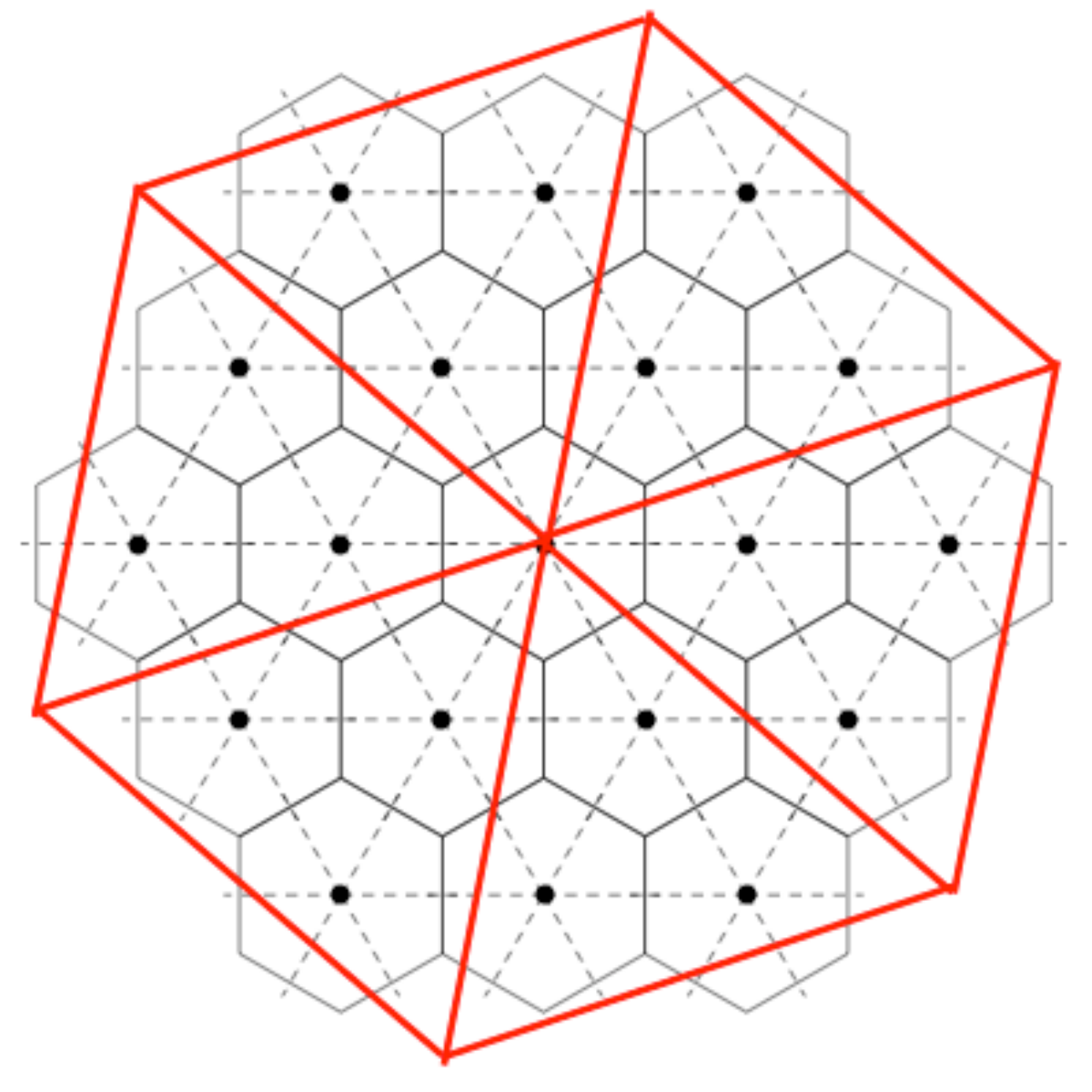}
\caption{Cells with the same colour. Any two cells of the same colour are at distance at least $\sqrt{\tfrac73}$}
\label{fig:7colors}
\end{figure}

 Rescale by multiplying by $\tfrac34$, 
 so that the diameter of a hexagonal cell is now $\tfrac34 \cdot \tfrac{2}{\sqrt{3}} = \sqrt{\tfrac34}$.  The distance between distinct rescaled cells corresponding to centres in $T'$ has now been reduced to at least $\tfrac34 \cdot 1.49 = 1.1175 > 1.1$, still bigger than 1.

Suppose that we are given any finite set of points in $\R^3$, take $r=1$, and let $G$ be the corresponding geometric graph.  Think of $\R^3$ as $\R^2 \times \R$.  Consider any cell $C$,  and let $G_C$ be the geometric graph corresponding to the points in the cylinder $C \times \R$, with threshold distance $r=1$.  We may now argue as in the proof of Theorem~\ref{thm.detR2}: for clarity we spell this out.
%
Observe that for $\bu, \bv \in C \times \R$, if $u_z=v_z$ then $d(\bu,\bv) \leq \tfrac34 <1$ so $uv \in E(G_C)$.
For $\bu, \bv \in C \times \R$ we write $\bu \prec \bv$ if $uv \in E(\overline{G_C})$ and 
$u_z < v_z$. 
If $\bu \prec \bv$ then
\[ 1< d(\bu,\bv)^2 = (u_x-v_x)^2 + (u_y-v_y)^2+(u_z-v_z)^2 \leq \tfrac34 + (u_z-v_z)^2,\]
and so $v_z > u_z +\tfrac12$.  If also $\bv \prec \bw$ then similarly  $w_z > v_z +\tfrac12$; and then $w_z > u_z+1$ and so $\bu \prec \bw$.  It follows that $\prec$ is a (strict) partial order, and $G_C$ is the co-comparability graph.
%
Thus, once more by the result of Duffus et al.~\cite{Duffus+}, we have $\chi_c(G_C) \leq 3$. 

Consider the 7-colouring of the cells.  For each colour $i=1,\ldots,7$ and each cell $C$ of colour $i$, properly clique colour the points in $C \times \R$ using colours $(i,1),(i,2),(i,3)$.  If two points in distinct cylinders have the same colour, then the distance between them is at least $1.1 >1$, 
so the corresponding vertices are not adjacent in $G$. Thus the colourings of the cylinders fit together to give a proper clique colouring of $G$ using at most 21 colours, as required.
\end{proof}

\smallskip

 The next result shows that, if we do not put some restriction on the dimension $d$, then we can say nothing about a geometric graph in $\R^d$. 

\begin{proposition} \label{prop.allgeom}
For each graph $G$ there is a positive integer $d$ such that $G$ is a geometric graph in $\R^d$, and indeed if $G$ has $n \geq 2$ vertices we can take $d \leq n-1$.
\end{proposition} 
\noindent 
Observe that the second part of this result follows immediately from the first, since the affine span of $n$ points has dimension at most $n-1$. 

\begin{proof}
We prove more, namely that for any $\eps>0$ there are points $\x^1,\ldots,\x^n$ in $\R^n$ such that for each $i$ we have ${\mathbf 1} \cdot \x^i =1$ and $\x^i$ is within distance $\eps$ of $\be^i$ (where $\be^i$ is the $i$th unit vector in $\R^n$), and such that for $i \neq j$
\[ d(\x^i,\x^j)
     \begin{cases} < \sqrt{2} & \text{if $ij$ is an edge }\\
                             > \sqrt{2} & \text{ if not.}
     \end{cases}
\] 

 The case $n=2$ is trivial.  Suppose that $n \geq 3$ and the result holds for $n-1$.  Start with $\x^i=\be^i$ for each $i=1,\ldots,n$.  We first adjust $\x^n$.
For $1\leq i<j\leq n$ let $z_{ij}$ be $-1$ if $ij$ is an edge and $+1$ if not.  Note that the $n$ $n$-vectors ${\mathbf 1}$ and $\be^n - \be^1, \be^n - \be^2,\ldots, \be^n - \be^{n-1}$ form a basis of $\R^n$.  Hence there is a unique vector $\y$ with $\y \cdot {\mathbf 1}=0$ and $\y \cdot (\be^n - \be^{i}) = z_{in}$ for each $i=1,\ldots,n-1$.

Let $\eps>0$, and assume (as we may) that  $\eps < 1/\| \y \|$.  Let $\delta= \eps/\|\y\|$, and re-set $\x^n$ to be $\be^n+ \delta \y$.
Note first that ${\mathbf 1} \cdot \x^n =1$ and $d(\x^n,\be^n) = \| \delta \y \| = \eps$.  For each $i \in [n-1]$ 
\[ \| \x^n - \be^{i}\|^2  =
 \| (\be^n - \be^{i}) + \delta \y \|^2  =  2 + 2 \delta z_{in} + \delta^2 \| \y \|^2 = 2 + \eps^2 + 2 \delta z_{in}. \]
But $\eps^2-2\delta = \eps(\eps- \frac2{\|\y\|})   < -\delta$. Thus $\| \x^n - \be^{i}\|^2$ is $<2 - \delta$ if $in$ is an edge and $>2+\delta$ if not.  Let
\[0< \eta < \min \{ \sqrt{2 + \delta} - \sqrt{2}, \sqrt{2} - \sqrt{2 - \delta} \}.\]
By the induction hypothesis, we may choose points $\x^1,\ldots,\x^{n-1}$ in $\R^{n}$ with $n$th co-ordinate 0, such that distances corresponding to edges are $<\sqrt{2}$ and other distances are $>\sqrt{2}$,
and for each $i \in [n-1]$ we have $\x^i \cdot {\mathbf 1}=1$ and $\x^i$ is within distance $\eta$ of $\be^i$.  By the triangle inequality,
$d(\x^n, \x^i) = d(\x^n, \be^i) + \eta_i$ for some $\eta_i$ with $| \eta_i| \leq \eta$.  Thus for each $i \in [n-1]$
\[ d(\x^n,\x^i)
     \begin{cases} < \sqrt{2-\delta}+ \eta < \sqrt{2} & \text{if  $in$ is an edge }\\
                             > \sqrt{2+ \delta} - \eta > \sqrt{2} & \text{ if not.}
     \end{cases}
\] 
This completes the proof by induction.
\end{proof}


Let $\chi_c^{\max}(n)$ be the maximum value of $\chi_c(G)$ over all $n$-vertex graphs.  Since the Ramsey number $R(3,k)$ satisfies $R(3,k)=\Theta(k^2/\log k)$, there exist $n$-vertex triangle-free graphs $G_n$ with stability number $O(\sqrt{n \log n})$  (see~\cite{Fiz} for the best known bounds) and thus with chromatic number and hence clique chromatic number $\Omega({\sqrt{n/\log n}})$. (Recall that $\chi_c=\chi$ for a triangle-free graph.)  Hence
\begin{equation} \label{eqn.chicmaxn}
  \chi_c^{\max}(n) = \Omega({\sqrt{n/\log n}}) \;\; \mbox{ as } \; n \to \infty.
\end{equation}
It now follows from Proposition~\ref{prop.allgeom} that 
\begin{equation} \label{eqn.bigd}
\chi_c^{\max}(\R^d) = \Omega(\sqrt{d/\log d}) 
\;\; \mbox{ as } \; d \to \infty.
\end{equation}
This shows explicitly that $\chi_c^{\max}(\R^d) \to \infty$ as $d \to \infty$, though the lower bound here is rather a long way from the upper bound (roughly $d^{d/2}$) provided by Proposition~\ref{thm.detRd}.
(See also Section~\ref{sec.concl}, where we discuss $\chi_c^{\max}(\R^d)$ in paragraph~(2), and $\chi_c^{\max}(n)$ in paragraphs~(5) and (6).)
\smallskip

It is convenient to give one more deterministic result here, which we shall use in the proofs in the next section and in the final section. For the sake of completeness, we include the straightforward proof.

\begin{lemma} \label{lem.detsmall}
Let the graph $G$ have $n$ vertices.  If $n \leq 5$ then $\chi_c(G) \leq 2$ except if $G$ is isomorphic to $C_5$ when $\chi_c(G)=3$.  If $n \leq 10$ then $\chi_c(G) \leq 3$. 
\end{lemma}
\begin{proof}
Suppose that $n \leq 5$.  If $\deg(v) \geq 3$ then colouring $N(v)$ with colour 1 and the other vertices with colour 2 shows that $\chi_c(G) \leq 2$: thus we may assume that each degree is at most 2.  If $G$ has a triangle then $G$ consists of a triangle perhaps with one additional disjoint edge, so $\chi_c(G) \leq 2$. If $G$ does not have a triangle, then either $G$  is isomorphic to $C_5$ or $\chi_c(G) \leq \chi(G) \leq 2$.  Also, since $C_5$ has no triangles, $\chi_c(C_5)=\chi(C_5)=3$.  This completes the proof of the first statement.

Now let us prove that $\chi_c(G) \leq 3$ for $n \leq 9$. Suppose for a contradiction that $n \leq 9$ and $\chi_c(G) >3$, and $n$ is minimal such that this can happen.  The minimum degree in $G$ is at least 3 (for if $\deg(v) \leq 2$ and $\chi_c(G-v) \leq 3$ then $\chi_c(G) \leq 3$).
 
Suppose that $\deg(v) \geq 4$ for some vertex $v$, and let $G'=G \setminus N(v)$.  Then $|V(G')| \leq 5$ and $\chi_c(G') \geq 3$ (since $\chi_c(G) \leq \chi_c(G')+1$).  Thus by the first part of the lemma, $\deg(v)=4$ and $G'$ is isomorphic to $C_5$.  But now $v$ has neighbours in $G'$, a contradiction.

It follows that $G$ is cubic.  Hence $n$ is even, and so $n \leq 8$.  Now let $v$ be any vertex and as before let $G'=G \setminus N(v)$.   Arguing as before, we must have $\chi_c(G') \geq 3$ so $G'$ is isomorphic to $C_5$ and $v$ has neighbours in $G'$, a contradiction.
\smallskip

It remains only to show that $\chi_c(G) \leq 3$ when $n =10$.  As above, we may assume that $G$ is connected and the minimum degree in $G$ is at least $3$.  If $G$ has a vertex $v$ with $\deg(v) \ge 5$, then 
by the case $n=4$ of the lemma,
$G'=G \setminus N(v)$ satisfies $\chi_c(G') \le 2$, since $G'$ consists of an isolated vertex and a 4-vertex graph:
but now, using the third colour for each vertex in $N(v)$, we see that $\chi_c(G) \le 3$. 
If each vertex has degree at most 3 then $\chi_c(G) \le \chi(G)\le 3$ by Brooks' theorem (since $G$ is connected and is not $K_4$).

Now we may assume, without loss of generality, that $G$ has a vertex $v$ with $\deg(v)=4$. Since $v$ is isolated in $G'=G \setminus N(v)$, by the case $n=5$ of the lemma,
$\chi_c(G') \le 2$ (and thus as before $\chi_c(G) \le 3$) unless $G'$ is the disjoint union of the vertex $v$ and the $5$-cycle $C = v_1,\ldots,v_5$ with edges $v_i v_{i+1}$ (where $v_6$ means $v_1$). Assume that $G'$ is indeed of this form.  We now have two cases.

\textbf{Case 1: }\textit{there are adjacent vertices $v_i$, $v_{i+1}$ in the cycle $C$ that form a triangle with some vertex $u \in N(v)$.} \\
We may $3$-clique colour $G$ as follows.  Without loss of generality, assume that $i=2$. Give colour 1 to $v$, $v_2$, $v_{3}$ and $v_5$; give colour 2 to $v_1$ and $v_{4}$; and give colour 3 to each vertex in $N(v)$.
Let $K$ be a monochromatic clique of size at least 2.  If $K$ has only colour 1, then $K$ cannot contain $v$ or $v_5$ (since they have no neighbours coloured 1), so we can add $u$ to $K$; $K$ cannot have only colour 2 (since the vertices coloured 2 form a stable set); and if $K$ has only colour 3 then we can add $v$ to $K$.

\textbf{Case 2: }\textit{no two vertices in the cycle $C$ form part of a triangle.} \\
Each vertex $u \in N(v)$ can be adjacent to at most two (non-adjacent) vertices in $C$, and every vertex $v_i$ in $C$ has at least 
$1$ and at most $2$ neighbours in $N(v)$.  Hence some vertex in $C$ has exactly one neighbour in $N(v)$: without loss of generality, assume that $v_1$ has exactly one neighbour, say, $u_1$ in $N(v)$.  Note that $u_1$ is not adjacent to $v_2$ or $v_5$: since $u_1$ is adjacent to at most one of $v_3, v_4$ we may assume, without loss of generality, that $u_1$ is not adjacent to $v_4$. Give colour 1 to $u_1, v_2,v_4$; give colour 2 to $v, v_3, v_5$; and give colour 3 to $v_1$ and each vertex in $N(v) \setminus \{u_1\}$. 

As in the first case, let $K$ be a monochromatic clique of size at least 2.  Then $K$ cannot be only colour 1 or only colour 2, since the vertices coloured 1 and the vertices coloured 2 both form stable sets; and if $K$ has only colour 3 then $v_1 \not\in K$ (since $v_1$ has no neighbours coloured 3) so we can add $v$ to $K$.
\end{proof}


The Gr\"{o}tzsch graph is triangle-free on $11$ vertices and has chromatic number $4$, and thus has clique chromatic number $4$.
Since $\chi_c^{\max}(10)=3$ by the last result, it follows that $\chi_c^{\max}(11)=4$. 
Indeed, we may deduce easily that 
\begin{equation} \label{eqn.chicmax4}
\chi_c^{\max}(n)=4 \;\; \mbox{ for } n=11,\ldots,16.
\end{equation}
In order to see it, suppose $G$ is connected and has $n \leq 16$ vertices: we must show that $\chi_c(G) \leq 4$.  If each vertex has degree at most 4 then $\chi(G) \leq 4$ by Brooks' Theorem, and so $\chi_c(G) \leq 4$.  If some vertex $v$ has degree at least 5 then $G'=G \setminus N[v]$ has at most 10 vertices, so $\chi_c(G) \leq 1+ \chi_c(G') \leq 4$.


\section{Random results}\label{sec.randproofs}

In this section we prove Theorem~\ref{thm.RGGsummary} and Proposition~\ref{prop.RGG4}. We use one preliminary lemma that concerns the appearance of small components in the random geometric graph $G$.  It is taken from Chapter 3 of~\cite{Penrose}, where it is proved using Poisson approximation techniques.
\begin{lemma} \label{lem.comps}
Let $k \geq 2$ be an integer, let $H$ be a connected unit disk graph with $k$ vertices, and let $\mu=\mu(H)>0$ be the constant defined in equation (3.2) in~\cite{Penrose}.
\begin{enumerate}
\item If $nr^{2(k-1)} \to 0$ then \whp\ $G$ has no component with $k$ or more vertices. 
\item If $nr^{2(k-1)} \to c$ where $0<c<\infty$ then the expected number of components isomorphic to $H$ tends to $\mu c$, and the probability that $G$ has  such a component tends to $1- e^{- \mu c}$.
\item If $nr^{2(k-1)} \to \infty$ and $r \to 0$ then \whp\ $G$ has a component $H$.
\end{enumerate}
\end{lemma}
\noindent
 (In part (2) above, the number of components isomorphic to $H$ in fact converges in distribution to Poisson$(\mu c)$.) 
We may now prove Theorem~\ref{thm.RGGsummary}, taking the parts in order.  We shall use the last lemma several times,  sometimes without explicit reference.

\medskip

\noindent
\emph{Proof of Theorem~\ref{thm.RGGsummary}}
\smallskip

\emph{Part} (1). The expected number of edges is asymptotic to $\binom{n}{2}\pi r^2/n \sim (\pi/2) \, nr^2$. Thus by Markov's inequality, if $nr^2 \to 0$ then \whp\ $G$ has no edges 
 so $\chi_c(G)=1$.  (This also follows from Lemma~\ref{lem.comps} part (1) with $H$ as the complete graph $K_2$.)
\smallskip

\emph{Part} (2).
If $nr^2 \to c$ where $0<c<\infty$, then the expected number of edges tends to $\mu c$, where $\mu=\mu(K_2)= \pi /2$ (edge-effects are negligible). Also, since $nr^4 \to 0$, \whp\ each component has size at most 2, and so $\chi_c(G) \leq 2$. Hence $\pr(\chi_c(G)=1) = \pr(G \mbox{ has no edges}) \to e^{-\mu c}$, and $\pr(\chi_c(G) = 2) \sim \pr(G \mbox{ has an edge}) \to 1- e^{-\mu c}$.
\smallskip

\emph{Part} (3).
If $nr^2 \to \infty$ then \whp\ $G$ has an edge (and indeed $G$ has at least one component that is an isolated edge), so $\chi_c(G) \geq 2$. If $nr^8 \to 0$ then \whp\ each component of $G$ has size at most 4, and then $\chi_c(G) \leq 2$ by Lemma~\ref{lem.detsmall}.  These two results combine to prove Part (3).
\smallskip

\emph{Part} (4).
If $nr^8 \to c$ (where $0<c<\infty$), then 
the probability there is a component $C_5$ tends to $1- e^{-\mu c}$, where $\mu=\mu(C_5)>0$. Also \whp\ $G$ has edges and each component has size at most 5.  Hence $\pr(\chi = 2) \sim \pr(G \mbox{ has no component } C_5) \to e^{-\mu c}$; and, using also Lemma~\ref{lem.detsmall}, $\pr(\chi = 3) \sim \pr(G \mbox{ has a component } C_5)  \to 1- e^{-\mu c}$.  
\smallskip

\emph{Part} (5).
If $nr^8 \to \infty$ and $r \to 0$, then 
\whp\ $G$ has a component $C_5$, and so $\chi_c(G) \geq 3$. The following lemma covers the remainder of the relevant range of values for $r$.

\begin{lemma}\label{lem.RGG2}
Let $G \in \RG$ with $(n/ \log n)\, r^8 \to \infty$ and $r \leq 0.46 \, \sqrt{\log n}$. Then $\chi_c(G) \geq 3$ \whp.
\end{lemma}
In order to simplify the proof of Lemma~\ref{lem.RGG2} we will make use of a technique known as Poissonization, which has many applications in geometric probability (see~\cite{Penrose} for a detailed account of the subject). Here we sketch all we need. Consider the related model of a random geometric graph $\RGuv$, where the set of points is given by a homogeneous Poisson point process of intensity $1$ in the square $\SR$ of area $n$. In other words, we form our graph from $N$ points in the square $\mathcal S_n$ chosen independently and uniformly at random, where $N$ is a Poisson random variable of mean $n$. 

The main advantage of generating our points by a Poisson point process arises from the following two properties: (a) the number of points that lie in any region $A\subseteq\SR$ of area $a$ has a Poisson distribution with mean $a$, and the numbers of points in disjoint regions of $\SR$ are independently distributed; and (b) by conditioning $\RGuv$ on the event $N=n$, we recover the original distribution of $\RG$. Therefore, since $\Pr(N=n)=\Theta(1/\sqrt n)$, any event holding in $\RGuv$ with probability at least  $1-o(n^{-\frac12})$ must hold \whp\ in $\RG$.

\begin{proof}[Proof of Lemma~\ref{lem.RGG2}]
Our plan is to show that \whp\ $G$ contains a copy of $C_5$ such that no edge of this copy is in a triangle in $G$, and so $\chi_c(G) \geq 3$.  In order to allow $r$ to be as large as  possible 
we consider a configuration of 5 points such that the corresponding unit disk graph is $C_5$, and the area $A$ that must contain no further points (to avoid unwanted triangles) is as small as possible.

We work in the Poisson model $\RGuv$. Within $\SR$ choose 
$(\lfloor \sqrt{n}/4r \rfloor)^2$
 disjoint square cells which are translates of $[0,4r)^2$.
For each of these cells, we shall consider a regular pentagon $Q$ centered at the center of the cell and contained well within the cell.

Consider first the square $[-2,2)^2$. Start with a regular pentagon, with extreme points listed clockwise as $\bv_1,\ldots,\bv_5$ around the boundary, centred on the origin ${O}=(0,0)$, and scaled so that the diagonals (for example $\bv_1 \bv_3$) have length 1. The angle $\bv_1 O \bv_2$ is $2 \pi/5$, and the line $O \bv_2$  is orthogonal to the line $\bv_1 \bv_3$ and bisects it.  Hence the radius (from the centre ${O}$ to each extreme point $\bv_i$) is $a:= |O \bv_1| = 1/(2\sin \frac{2\pi}{5})  \approx 0.525731$. (We give numbers rounded to 6 decimal places.) If $T$ is the midpoint of the side $\bv_1 \bv_2$, then the line $OT$ is orthogonal to $\bv_1 \bv_2$ and the angle $\bv_1 OT$ is $\pi/5$.  Hence the side length~$s$ (the length of $\bv_1 \bv_2$ for example) satisfies $\frac{s}{2a} = \sin ( \pi/5)$, so $s= \sin(\pi/5)/\sin(2\pi/5) = 1/(2\cos \frac{\pi}{5}) \approx 0.618034$. For each successive pair $\bv_i \bv_{i+1}$ of extreme points (including $\bv_5 \bv_1$), let $B_i$ be the intersection of the unit radius disks centred on $\bv_i$ and $\bv_{i+1}$; and let the `controlled region' $B$ be the union of the $B_i$, with area~$A$.  For the value of $A$, we have the following claim. 

\begin{figure}[h]
\begin{center}
\includegraphics[width=3in,height=3in]{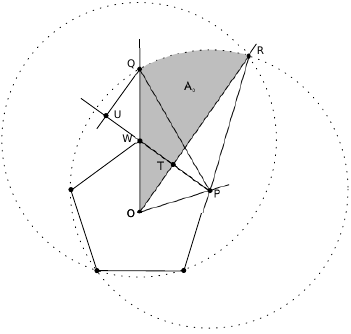}
\end{center}
\caption{Configuration of points in the proof of the claim}
\label{fig3}
\end{figure}

\medskip

\noindent
\textbf{Claim:} $A \approx 4.633376$.

\begin{proof}[Proof of the claim]
We may calculate $A$ as follows.
Let us take $\bv_1$ to be on the $y$-axis above the origin $O$, so $\bv_1=(0,a)$. 
Now $\bv_2=(a \cos \frac{\pi}{10}, a \sin \frac{\pi}{10})$.
Let us denote $\bv_1$ by $W$ and $\bv_2$ by $P$.

Suppose that the circle $C$ of radius 1 centred on $P$ meets the lines $x=0$ (on which $W$ lies) and $y= (\tan\frac{3 \pi}{10}) x$ (bisecting the angle between $OW$ and $OP$) above the $x$-axis at $Q$ and $R$ respectively.  Then the area $A$ is $10 A_0$, where $A_0$ is the area bounded by these two straight lines and the arc of the circle $C$ between $Q$ and $R$ -- see the shaded area on Figure~\ref{fig3}.  We may calculate $A_0$ as the area $A_0^1$ of the sector of the circle bounded by the arc between $Q$ and $R$ and the radii $PQ$ and $PR$, less the area $A_0^2$ of triangle $OPR$, plus the area $A_0^3$ of triangle $OPQ$. Recall that $T$ is the point of intersection of the lines $PW$ and $OR$, and note that $|PT|=|WT|=s/2$ and $PW$ and $OR$ are orthogonal.
Now (by Pythagoras' theorem) $\, |OT|^2=a^2 - (s/2)^2$ and $|TR|^2= 1-(s/2)^2$; and so $|OT| = a \sqrt{1-\sin^2 (\pi/5)} \approx 0.425325$ and $|TR| = \sqrt{1 - a^2 \sin^2 (\pi/5)} \approx 0.951057$.

Drop a perpendicular from $Q$ to the (extended) line $PW$, meeting the line at~$U$. Note that the angle $UWQ$ is $\frac{3\pi}{10}$, and so $|UW|=|QU| \cot(3\pi/10)$. Hence, by considering the triangle $PQU$  in which $|PQ|=1$, $|UQ|$ is the positive solution $h$ of the quadratic equation 
\[ \big( s + \cot(3\pi/10) h \big)^2 + h^2 = 1;\]
and thus we obtain $h \approx 0.406737$. It follows that the angle $WPR$ is $\alpha = \arcsin |TR| \approx 1.256637$, the angle $WPQ$ is $\gamma = \arcsin |UQ| \approx 0.418879$, and so
$$
A_0^1 = \frac {\alpha - \gamma}{2} \approx 0.418879. 
$$
Moreover, 
$$
A_0^2 = \frac {s |OT|}{4} + \frac {s |TR|}{4} \approx 0.212663,
$$
$$
A_0^3 = \frac {s |UQ|}{2} + \frac {s |OT|}{2} \approx 0.257121,
$$
and so $A = 10 A_0 = 10 (A_0^1 - A_0^2 + A_0^3) \approx 4.633376$.
\end{proof}

\par

\medskip

We continue with the proof of Lemma~\ref{lem.RGG2}. Let $0<b<A^{-\frac12} \approx 0.464570$; and let $r=r(n)$ satisfy $(n/ \log n)\, r^8 \to \infty$ and $r \leq b \sqrt{\log n}$.  As indicated earlier, we shall show that \whp\ $G$ contains a copy $J$ of $C_5$ such that no edge of $J$ is in a triangle in $G$, and so $\chi_c(G) \geq 3$.

Choose $\eps>0$ sufficiently small that $\eta := 1- (1+\eps)^2 b^2 A >0$, the region $(1+\eps)B$ is contained in the ball centred on $O$ with radius $2$, 
and $\eps < (1+\eps) s < 1-\eps$.
Scale up by a factor $1+\eps$,  and use the notation $\bv'_i$, $A'$, $B'$ to refer to the rescaled case.
Note that 
$B'$ is contained in $[-2,2)^2$ (by our assumption on $(1+\eps) B$).  Put small open balls of radius $\eps/2$ around the five extreme points $\bv'_i$ of the pentagon, and note that these small balls are all disjoint (since $(1+\eps) s > \eps$).
 If $\x$ and $\y$ are points in the small balls at non-adjacent vertices $\bv'_i$ and $\bv'_j$ then $d(\x, \y) > 1$ (since $d(\bv'_i,\bv'_j)=1+\eps$).  If $\x$ and $\y$ are points in the small balls at adjacent vertices $\bv'_i$ and $\bv'_{i+1}$ (where $\bv'_{6}$ means $\bv'_1$) then $d(\x, \y) < 1$ (since $(1+\eps)s + \eps<1$); and if ${\bf z} \not\in B'$ then either $d({\bf z},{\bf x})> d(\z, \bv'_i) - \eps/2 >1$ or similarly $d({\bf z},{\bf y})> 1$,
so we do not get triangles involving a point $\z \not \in B'$.

 
 Now rescale by $r$, and call the rescaled controlled region $B''$.  Note that the area of $B''$ is $(1+\eps)^2 r^2 A$. 
 If exactly one Poisson point $\x$ lies in each rescaled small ball and there are no other such points in $B''$ then we have a copy of $C_5$ as desired. 
Setting $\lambda = \pi(\eps r /2)^2$, the probability $q_n$ of this happening satisfies
\[ q_n = (\lambda e^{-\lambda})^5 e^{-r^2 ( (1+\eps)^2 A - 5 \pi (\eps/2)^2)} = \lambda^5 e^{-r^2  (1+\eps)^2 A}.\] 
Since events within different cells are independent, the probability $p_n$ that $G \in \RGuv$ has no $C_5$ as desired satisfies
\[p_n \leq (1-q_n)^{(\lfloor \sqrt{n}/4r \rfloor)^2}
 \le \exp\left(-(1+o(1))\frac{q_n \,  n}{16r^2} \right). \] 
Observe that 
\[ r^2  (1+\eps)^2 A \leq b^2 (1+\eps)^2 A \log n = (1-\eta) \log n. \]
If $1 \leq r \leq b \sqrt{\log n}$ then 
\[ q_n n /r^2 = \lambda^5 e^{-r^2 (1+\eps)^2 A} \, n/r^2= \Omega(r^8 n^{\eta}) = \Omega(n^{\eta}), \]
and if 
$(\frac{\log n}{n})^{\frac18} \ll r \leq 1$ then
\[ q_n n /r^2 =  \Omega(\lambda^5 n/r^2)= \Omega(r^8 n) \gg \log n.  \] 
Thus in both cases $p_n=o(n^{-\frac12})$.  It follows that the failure probability in the original $\RG$ model is $o(1)$, as required.
\end{proof}
\medskip

\emph{Part} (6) (of Theorem~\ref{thm.RGGsummary}).
The next lemma proves Part (6), and thus completes the proof of Theorem~\ref{thm.RGGsummary}.

\begin{lemma}\label{lem.RGG}
Let $G \in \RG$ with 
$r \geq 9.27 \, \sqrt{\log n}$.  Then $\chi_c(G)=2$ \whp.
\end{lemma} 


\begin{proof}[Proof of Lemma~\ref{lem.RGG}]
Clearly $G$ has an edge \whp, and so $\chi_c(G) \geq 2$ \whp. Hence, we only need to show that $\chi_c(G) \leq 2$ \whp.  As in the proof of Proposition~\ref{thm.detR3}, start with a hexagonal packing in the plane, as in Figure 1, formed from the Voronoi cells (with vertical left and right sides) of the unit triangular lattice $T$ (where the edges have unit length).  The hexagonal cells have area $\sqrt{3}/2$ and diameter $2/\sqrt{3}$. 

Now rescale by the factor $\sqrt{ \frac{2}{\sqrt{3}}(1+\eps) \log n}$ for some suitably small $\eps > 0$. As a result, each cell has area $(1+\eps) \log n$ and diameter 
$\delta:= ((1+\eps)8 \log n)^{\frac12} 3^{-3/4}$.  (For orientation, note that the lower bound on $r$ is (for small $\eps$) more than $7.4 \, \delta$.) By shrinking slightly in the $x$ and the $y$ directions, we may ensure that the left and right sides of the square $\SR$ lie along vertical sides of cells (more precisely, we may ensure that, as we move up the left side of the square, every second internal cell has its vertical left boundary along the side of the square, and every second one is bisected by the side of the square; and similarly for the right side of the square),
and each cell which meets a horizontal side of $\SR$ is at least half inside $\SR$.  We then obtain (at least for large $n$) a partition of the square $\SR$ such that each cell has diameter at most $\delta$, each `internal' cell not meeting the boundary has area at least $a= (1+\eps/2) \log n$, and each `boundary' cell meeting the boundary has area at least $a/2$. 
There are $O(n/\log n)$ internal cells and $O(\sqrt{n/\log n})$ boundary cells.

The probability that a given internal cell contains at most one point in its interior is at most 
\[ \left( 1 - \frac {a}{n} \right)^n + n \left( 1 - \frac {a}{n} \right)^{n-1} \frac {a}{n} \le n^{-1-\eps/2+o(1)}. \]
Since there are $O(n/\log n)$ internal cells, the expected number of such cells is $n^{-\eps/2+o(1)} = o(1)$. Similarly, the probability that a given boundary cell contains at most one point in its interior is at most $n^{-\frac12-\eps/4+o(1)}$; 
and since there are $O(\sqrt{n/\log n})$ boundary cells, the expected number of such cells  is $n^{-\eps/4+o(1)} = o(1)$.
It follows from Markov's inequality that \whp\ all cells have at least two points in their interior. 

It suffices now to show (deterministically) that for each set of points in $\SR$ with at least two in the interior of each cell, the corresponding graph $G$ has $\chi_c(G) \leq 2$. To do this, we colour the vertices of $G$ arbitrarily as long as both colours are used in every cell: we shall show that this gives a proper clique-colouring. 

Observe that $G$ has no isolated vertices since $r$ is more than the diameter $\delta$ of a cell
%
%
(indeed, $r > 7.4 \delta$ and so -- assuming $n$ is large -- the minimum degree may be shown to be at least 95, since in the triangular lattice there are 48 lattice points $(x,y) \geq (0,0)$ within graph distance 7 of $(0,0)$, and thus within Euclidean distance $6.1 \delta$,  so each point in each of these cells is at Euclidean distance $< r$ from each point in the cell corresponding to $(0,0)$).
%
Consider any maximal clique $K$ in $G$ with corresponding Euclidean diameter $D$ (so $0< D \le r$), and suppose that $D$ is attained for the Euclidean distance between the points $\bu$ and $\bv$ corresponding to vertices $u$ and $v$ in $K$. Let $\y$ be the midpoint of the line joining $\bu$ and $\bv$, and let the cell $C$ contain $\y$.
Since for each vertex $w$ in the clique $K$, the corresponding point $\bw$ is at distance at most $D$ from both $\bu$ and $\bv$, it follows that $\bw$ is at distance at most $\sqrt{D^2 - (D/2)^2} = \sqrt{3} D /2$ from $\y$. Hence if
\begin{equation}\label{eq:Dbound}
\sqrt{3} D /2 + \delta \le r,
\end{equation}
then every point of the cell $C$ is at distance at most $r$ from all points of $K$.  Since $K$ is maximal, all vertices corresponding to points that belong to the cell must be in $K$, 
 and so $K$ is not monochromatic. Since $D \le r$, the desired inequality~\eqref{eq:Dbound} holds as long as $\sqrt{3} r/2 + \delta \le r$, which is equivalent to $r \geq 4 \delta (1+ \sqrt{3}/2)$. But
\[ 4 (1+ \sqrt{3}/2) 8^{\frac12} 3^{-3/4} = (1+ \sqrt{3}/2) 2^{7/2}3^{-3/4} = 9.261506  \]
to 6 decimal places.  Thus, by choosing $\eps>0$ sufficiently small, we see that it suffices to have $r \geq 9.2616 \sqrt{\log n}$. 
\end{proof}

We have completed the proof of Theorem~\ref{thm.RGGsummary}. It remains to prove Proposition~\ref{prop.RGG4}. The first part of that result follows directly from Lemmas~\ref{lem.detsmall} and~\ref{lem.comps}, since we already know that $\chi_c(G) \geq 3$ \whp, and the latter lemma shows that \whp\ $G$ has no components with more than 10 vertices.  The second part will follow easily from the next lemma, by considering a connected geometric graph $H$ such that $\chi_c(H)= \chi_c^{\max}(\R^2)$.  
\begin{lemma} \label{lem.allcomps}
Let $h \geq 2$ and let $H$ be any given connected geometric graph with $h$ vertices.  
Suppose that $n r^{2(h-1)} \to \infty$ and $r \leq \sqrt{\log n / (\pi h)}$. 
Then for $G \in \RG$, \whp\ $G$ has a component isomorphic to $H$.
\end{lemma}
\begin{proof}
If $nr^{2(h-1)} \to \infty$ and $r \to 0$ then \whp\ $G$ has a component $H$ by Lemma~\ref{lem.comps}.  To handle larger values of $r$, we now work in the Poisson model $\RGuv$.  Assume from now on that $r \geq 1/\log n$ say (and still $r \leq \sqrt{\log n / (\pi h)}$). 
Fix distinct points $\x_1,\ldots,\x_h$ such that, for each distinct $i$ and $j$, $d(\x_i, \x_j)<1$ if $ij \in E(H)$ and $d(\x_i, \x_j)>1$ if $ij \not\in E(H)$.  Thus the unit disk graph generated by these points is $H$.  Let $\alpha = \max \{ d(\x_i, \x_j): ij \in E(H)\}$, let $\beta= \min \{ d(\x_i, \x_j): ij \not\in E(H)\}$, and let $\gamma= \min\{ d(\x_i, \x_j): i \neq j\}$.  Let $0 < \eta \leq \frac12 \min \{1-\alpha, \beta-1, \gamma\}$. Put a small open ball $B(\x_i,\eta)$ of radius $\eta$ around each point $\x_i$.
Observe that these balls are pairwise disjoint, and if $\y_i \in B(\x_i, \eta)$ for each $i$ then $\y_1,\ldots,\y_h$ yield the same geometric graph~$H$.
 
Let $C_1$ be the set of points within distance 1 of the points $\x_i$ (so $C_1$ is the union of the balls $B(\x_i,1)$), 
and let $A_1$ be the area of $C_1$.  Observe that $A_1< \pi h$ since $h \geq 2$ and $H$ has an edge.
Let $C$ be the set of points within distance $1+\eta$ of the $\x_i$, and let $C$ have area $A$.
Let $b = (\pi h)^{-1/2}$.  If $\eta$ is chosen sufficiently small then $b^2 A<1$; assume we have done this.

 If each ball $B(\x_i,\eta)$ contains exactly one Poisson point and there are no other such points in $C$, then we have a copy of $H$ forming a component of $G$. 
Now scale by $r$, note that we can pack $\Theta(n/r^2)$ disjoint copies of the configuration in $\mathcal S_n$, and we may argue as in the proof of Lemma~\ref{lem.RGG2}, as follows.

Set $\lambda = \pi (\eta r)^2$.  Let $q_n$ be the probability that each small ball contains exactly one Poisson point and there are no such points where they should not be.  Then
\[ q_n = (\lambda e^{-\lambda})^h e^{-r^2 (A - h \pi \eta^2)} = \lambda^h e^{-r^2 A}.\]
Since events within different cells are independent, for some constant $c>0$ the probability $p_n$ that $G \in \RGuv$ has no component $H$ satisfies
\[p_n \leq (1-q_n)^{\frac{cn}{r^2}} \le \exp\left(-\frac{c n q_n}{r^2} \right). \] 
Now, for $r \le b\sqrt{\log n}$, we have $r^2 A \le b^2 A \log n$,  and so
\[
 q_n n /r^2 = \lambda^h e^{-r^2 A} n/r^2 = \Omega(n^{1-b^2 A} r^{2h-2})= \Omega(n^{1-b^2 A} (\log n)^{-(2h-2)}).
\] 
Thus, since $b^2 A <1$, we have $p_n=o(n^{-\frac12})$.  It follows that the failure probability in the original $\RG$ model is $o(1)$, as required. 
\end{proof}


\section{Concluding Remarks} \label{sec.concl}

Let us pick up a few points for further thought.
\smallskip 

\begin{enumerate}
\item  Recall that $\chi_c^{\max}(\R^2)$ is the maximum value of $\chi_c(G)$ over geometric graphs $G$ in the plane, and we saw that $3 \leq \chi_c^{\max}(\R^2) \leq 9$. Can we improve either bound?

Observe that if a geometric graph $G$ is triangle-free then $G$ is planar (if in the embedding of a geometric graph two edges cross, then this induces a triangle in $G$, see for example~\cite{Breu}) and so $\chi_c(G) \leq \chi(G) \leq 3$ by Gr\"{o}tzsch's theorem. We saw in Lemma~\ref{lem.detsmall} that $\chi_c(G) \leq 3$ for all graphs with at most 10 vertices. The Gr\"{o}tzsch graph showed that this bound does not extend to $n=11$ (see also equation~(\ref{eqn.chicmax4}), and point (4) below). But the Gr\"{o}tzsch graph is not a geometric graph in the plane, 
so perhaps the upper bound 3 extends to larger values $n$ when we restrict our attention to geometric graphs? Any extension for geometric graphs would lead to an improvement in Proposition~\ref{prop.RGG4} Part~(1). If it turns out that $\chi_c^{\max}(\R^2)=3$, then Theorem~\ref{thm.RGGsummary} is tighter than it currently seems, and Proposition~\ref{prop.RGG4} is redundant. If $\chi_c^{\max}(\R^2) > 3$ then it would be interesting to refine Part~(5) of Theorem~\ref{thm.RGGsummary}.

\smallskip

\item
More generally, can we say more about $\chi_c^{\max}(\R^d)$?
We saw in Proposition~\ref{thm.detR3} that $\chi_c^{\max}(\R^3) \leq 21$: can we improve this upper bound?  Can we find a geometric graph in $\R^3$ with $\chi_c(G)>3$? 
We have seen that  $\chi_c^{\max}(\R^d)$ is at most $2e^{2\sqrt{d}} \, d^{d/2}$ and is $\Omega(\sqrt{d/ \log d})$ as $d \to \infty$.  Can we improve these bounds?

\textit{Remark:} After submission of this paper the upper bound on $\chi_c^{\max}(\R^d)$ was improved in~\cite{Fox} to $2^{O(d)}$.
\smallskip

\item
 In the light of the last two parts of Theorem~\ref{thm.RGGsummary} (and Proposition~\ref{prop.RGG4}), it is natural to ask if there is a constant $\rho$, where $0.46 \leq \rho < 9.27$, 
such that for $G \in \RG$ and any $\varepsilon>0$, we have \whp 
\[ \chi_c(G)  \begin{cases}
\geq 3 & \mbox{ if } \; n^{-1/8} \ll r \leq (\rho - \varepsilon) \sqrt{\log n}\\
= 2 & \mbox{ if }  \; r \geq (\rho + \varepsilon) \sqrt{\log n}.
\end{cases} \]
\smallskip

\item
 Recall that $\chi_c^{\max}(n)$ is the maximum value of $\chi_c(G)$ over all $n$-vertex graphs.  Trivially $\chi_c^{\max}(1)=1$.  We saw in Lemma~\ref{lem.detsmall} and equation~(\ref{eqn.chicmax4}) that
\[  \chi_c^{\max}(n) =
\begin{cases}
2 & \mbox{ if } n=2, 3, 4\\
3 & \mbox{ if } n=5,\ldots, 10\\
4 & \mbox{ if } n=11,\ldots,16
\end{cases}
\]
What about larger values of $n$?

Now consider asymptotic behaviour. We saw in equation~(\ref{eqn.chicmaxn}) that $\chi_c^{\max}(n) = \Omega(\sqrt{n/\log n})$.
On the other hand, we claim that
\begin{equation} \label{eqn.chicmaxub} 
\chi_c^{\max}(n) \le 2\sqrt{n}.
\end{equation}
We may see this as follows.  \emph{Repeatedly}, pick greedily  a maximal independent set, give all the vertices in the set 
the same fresh colour and remove them, \emph{until} we find a maximal independent set $I$ of size less than $\sqrt{n}$.  Such a set $I$ is a dominating set in the remaining graph $H$, so $\chi_c(H) \leq |I|+1$, see~\cite{BGGPS2004,mmp-gnp}. Thus if $H$ has $h$ vertices, then at most $\min\{ |I|+1, h \}$ further colours are needed.

In the first phase we use at most  $(n-h)/\sqrt{n}= \sqrt{n}- h/\sqrt{n}$ colours.  If $h \geq \sqrt{n}$ then we use at most $(\sqrt{n}-1)+(|I|+1) < 2 \sqrt{n}$ colours in total.  If $h < \sqrt{n}$ then we use at most $\sqrt{n} + h < 2 \sqrt{n}$ colours, and hence $\chi_c^{\max}(n) < 2\sqrt{n}$.  This proves the claim~(\ref{eqn.chicmaxub}).

We know that $\chi_c^{\max}(n)$ is $\Omega(\sqrt{n/\log n})$ and $O(\sqrt{n})$. Can we say more about the asymptotic behaviour of $\chi_c^{\max}(n)$?  See also~\cite{ErdGT}, and Problem~1 there in particular.

Is it true that for each $n$, $
\chi_c^{\max}(n)$ 
is achieved by a triangle-free $n$-vertex graph?  Indeed, could it even be the case that every graph has a triangle-free subgraph with at least the same value of~$\chi_c$? 
\smallskip

\item Our upper bound on $\chi_c(G)$ gives an upper bound on the \emph{clique transversal number} $\tau_c(G)$, which is defined to be the minimum size of a set $S$ of vertices which meets all maximal cliques (ignoring isolated vertices).  For each $n$-vertex graph $G$, since the maximum size of a set of vertices containing no maximal clique is at least $n/\chi_c(G)$, we have
\[ \tau_c(G) \leq n - n/\chi_c(G). \]
The result noted above that $\chi_c(G) \leq 2 \sqrt{n}$ yields $\tau_c(G) \le n-\frac12 \sqrt{n}$, which may be compared with the best known bound $\tau_c(G) \le n-\sqrt{2n}+\sqrt{2}$ (see~\cite{ErdGT}). It is not likely to be easy to improve our upper bound by say a factor $4$ to $\chi_c(G) \leq (1/2) \sqrt{n}$, since that would strictly improve the upper bound on $\tau_c$ (to $\tau_c(G) \le n-2\sqrt{n}$).
\smallskip

\item
Finally, consider the number of dimensions we need to embed a graph.
Let $d^*(n)$ be the least value $d$ such that every graph with $n$ vertices is geometric in $\R^d$.  Then $d^*(n) \leq n-1$ 
by Proposition~\ref{prop.allgeom}. 
We claim that
\begin{equation} \label{eqn.dstar}
  d^*(n) = \Omega(\log n/\log \log n). 
\end{equation}
For, let $\eps>0$, and let $f(n)=(1-\eps) \log n / \log\log n$ for $n \geq 3$.  Suppose that $d^*(n) \leq f(n)$ for arbitrarily large values $n$.  We shall obtain a contradiction.

Note first that $f(n)$ is increasing for $n \geq 16$.
Define $m=m(d) = \lceil d^{(1+\eps)d} \rceil$.  Clearly $m(d) \geq 16$ for $d \geq 3$.  Now $f(m(d+1)) \sim (1-\eps^2) d$ as $d \to \infty$: hence, for some constant $d_0 \geq 3$ we have $f(m(d+1)) \leq d$ for each $d \geq d_0$.

Let $d_1 \geq d_0$ be arbitrarily large.  There exists $d \geq d_1$ such that $m(d) \leq n < m(d+1)$ for some $n$ with 
$d^*(n) \leq f(n)$.  Now
\[ d^*(m(d)) \leq d^*(n) \leq f(n) \leq f(m(d+1)) \leq d,\]
and so
\[ \chi_c^{\max}(\R^d) \geq \chi_c^{\max}(m(d)).\]
But by~(\ref{eqn.chicmaxn}), for some constant $c>0$ we have
\[\chi_c^{\max}(n) \geq c \sqrt{n/\log n} \;\; \mbox{ for each } n \geq 3.\]
Hence, 
 \[ \chi_c^{\max}(\R^d) \geq 
 c \, d^{(1+\eps)d/2} (\log(m(d))^{-1/2} \gg d^{(1+\eps/2)d/2}.\]
But this contradicts the upper bound on $\chi_c^{\max}(\R^d)$
in Proposition~\ref{thm.detRd}, and so we have established the claim~(\ref{eqn.dstar}).

Now we know that $d^*(n) = \Omega(\log n/\log \log n)$ and $d^*(n) \leq n-1$.  Our bounds are wide apart. What more can be said about $d^*(n)$?

%
\bigskip

\noindent
{\bf Acknowledgements}

Thanks to Mike Saks, Lena Yuditsky and Shira Zerbib for helpful discussions, and thanks to an anonymous reviewer 
whose comments much improved the paper.

\end{enumerate}


\end{document}